\documentclass[12pt,a4paper,oneside,toc=bibliography]{amsart}
\usepackage[left=1.1in, right=0.8in, top=1.0in, bottom=1.0in]{geometry}
\usepackage{graphicx}
\usepackage{setspace}
\usepackage{indentfirst} 
\usepackage{cases}
\usepackage{wrapfig}
\usepackage[toc,page]{appendix}
\usepackage{amsmath}
\usepackage{amsthm}
\usepackage{amssymb}
\usepackage{enumerate}
\usepackage{mathrsfs}
\usepackage{dcolumn}
\usepackage{thmtools}
\usepackage{thm-restate}
\usepackage{hyperref}
\usepackage[capitalize,nameinlink]{cleveref}
\usepackage{mathtools}

\newcolumntype{L}{D{.}{.}{2,5}}
\theoremstyle{plain}
\newtheorem{thm}{Theorem}
\newtheorem{hypo}{Conjecture}
\newtheorem{proposition}{Proposition}
\newtheorem{lemma}{Lemma}
\newtheorem{corollary}{Corollary}
\newtheorem{remark}{Remark}
\newtheorem{mydef}{Definition}

\DeclareMathOperator{\inc}{inc}

\DeclareMathOperator{\sign}{sign}

\DeclareMathOperator{\psh}{PSH}
\DeclareMathOperator{\loc}{loc}

\begin{document}

\title[Paley problem for plurisubharmonic functions]{\textbf On sharp constants in Paley problem for \\ Plurisubharmonic Functions of lower Order $\rho>1$}
\author{\textbf{A. B\"erd\"ellima}}
\address{Institute of Mathematics}
\address{Technische Universit\"at Berlin}\address{10623 Berlin, Germany}
 \thanks{Electronic address: \texttt{berdellima@math.tu-berlin.de/berdellima@gmail.com},}
\maketitle

\begin{abstract}
In 1999 Khabibullin established the best estimate in Paley problem for a plurisubharmonic function $u$ of finite lower order $0\leq\rho\leq 1$. For $\rho>1$ obtaining a sharp estimate has remained an open question. In this work we solve this problem. We also provide some estimates for the types of the characteristic functions $T(r,u)$ and $M(r,u)$.  
\end{abstract} 




\section{Introduction}
\subsection{Plurisubharmonic functions}
A function $u$ with values in $[-\infty,\infty)$ is called plurisubharmonic (psh) if 
\begin{enumerate}[(i)]
\item $u$ is upper semicontinuous and $u\not\equiv -\infty$;
\item for every $r\geq 0$ and every $z,w\in\mathbb C^n$ it holds that
\begin{equation}
	\label{eq:psh}
	u(z)\leq\frac{1}{2\pi}\int^{2\pi}_0u(z+re^{i\theta}w)\,d\theta.
\end{equation}
\end{enumerate}
The second condition is equivalent to saying that the mapping $\tau\mapsto u(z+\tau\,w)$ is subharmonic in $\mathbb C$ for any fixed $z,w\in\mathbb C^n$.
The set of psh functions forms a convex cone in the vector space of semicontinuous functions. If $u_1, u_2,...$ is a decreasing sequence of psh functions, then $u^*=\lim_{n\to\infty}u_n$ is also a psh function (\cite[p.225-226]{Hormander1}). 
For a function $u$ let $M(r,u)\coloneqq\max\{u(z):|z|=r,z\in\mathbb{C}^n\}$ and $u^+(z)=\max\{u(z),0\}$. Define
\begin{equation}
 \label{Nev-main}
 T(r,u)\coloneqq\int_{\mathbb{S}^n}u^+(r\zeta)\,ds_n(\zeta)
 \end{equation}
 to be the Nevalinna characteristic of $u$ where $\mathbb{S}^n$ is the unit sphere in $\mathbb{C}^n$ and $s_n$ is the rotation-invariant (\cite{Rudin}, p.12) positive Borel measure on $\mathbb{S}^n$ such that $s_n(\mathbb{S}^n)=1$. We say a function $u$ is of {\em lower order} $\rho>0$ if
 \begin{equation}
  \label{lowerorder}
  \liminf_{r\to+\infty}\frac{\log T(r,u)}{\log r}=\rho.
 \end{equation}
 \subsection{Paley problem}
An important issue in the theory of functions of complex variables is estimation of the following quantity 
 \begin{equation}
 \label{eq:M-T}
 \vartheta(u)\coloneqq\liminf_{r\to+\infty}\frac{M(r,u)}{T(r, u)}
 \end{equation}
for a function $u$ defined on $\mathbb C^n, n\geq 1$. Calculation of $\vartheta(u)$, when $u$ is an entire function in $\mathbb C$ of finite or infinite order was first considered by \cite[Paley-1932]{Paley}. The estimate $\vartheta(u)$ was known to Paley for the range $0\leq\rho\leq1/2$. 
The case $\rho>1/2$ remained an open question for many years and eventually it became known as Paley problem. This problem was conclusively proved in 1967 by Govorov \cite{Govorov}. Petrenko \cite{Petrenko} extended Paley problem to meromorphic functions of finite lower order. An analogue estimation of $\vartheta(u)$ for subharmonic functions $u$ of finite lower order in $\mathbb{R}^m,m\geq 2$, was obtained by
Dahlberg \cite{Dahlberg}. Generalizations for the estimation of \eqref{eq:M-T} in terms of the $L^p$ metric for $1\leq p<+\infty$ for subharmonic functions in $\mathbb{R}^m$ were also obtained by Sodin \cite{Sodin} for $m=2$. Kondratyuk, Tarasyuk, and Vasyl'kiv \cite{Kondratyuk} obtained respective generalizations for $m>2$. Eventually the problem of Paley was carried over to complex functions of several variables.  Khabibullin considered it for meromorphic functions in $\mathbb{C}^n$ for $n>1$ and later posed the problem of finding the best analogue estimate of \eqref{eq:M-T} for the class of psh functions and entire functions of severable variables \cite[Khabibullin]{Kh95, Kh99}. For $n>1$ define 
$$ P_n(\rho)\coloneqq\left\{
\begin{array}{lr}
\displaystyle\frac{\pi\rho}{\sin(\pi\rho)}\prod_{k=1}^{n-1}\Big(1+\frac{\rho}{2k}\Big) & : 0\leq\rho\leq\displaystyle\frac{1}{2}\\
\pi\rho \displaystyle\prod_{k=1}^{n-1}\Big(1+\frac{\rho}{2k}\Big)& : \rho>\displaystyle\frac{1}{2}\end{array}
\right.$$
More precisely Khabibullin showed:
\begin{thm}\cite[Theorem 1]{Kh99}
	\label{Khthm}
	Let $u$ be a psh function in $\mathbb{C}^n$ of finite lower order $\rho$. Then
	\begin{align}
	\label{Kh99} 
	\vartheta(u)\leq \left\{
	\begin{array}{lr}
	\displaystyle P_n(\rho) & : 0\leq\rho\leq\displaystyle 1\\
	e^{n-1}P_n(\rho) & : \rho>\displaystyle 1
	\end{array}
	\right. 
	\end{align} 
	The estimate for $0\leq\rho\leq 1$ is best possible. 
\end{thm} 
 
\subsection{Khabibullin's conjecture} 
In connection with Paley problem for psh functions, Khabibullin \cite{Kh02} has made the following conjecture:
\begin{hypo}[Khabibullin's conjecture]
	\label{Khbconj}
	Let $S(t)$ be a nonnegative increasing function on $[0,+\infty)$, and convex with respect to $\log t$. Further let $\rho>1$ and $n\geq 2,n\in\mathbb{N}$. If \begin{equation}
	\label{eq:conj1}
	\int^1_0S(tx)(1-x^2)^{n-2}\,dx\leq t^{\rho}\hspace{0.3cm}t\geq 0
	\end{equation}
	then
	\begin{equation}
	\label{eq:conj2}
	\int^{+\infty}_0S(t)\frac{t^{2\rho-1}}{(1+t^{2\rho})^2}\,dt\leq \frac{\pi(n-1)}{2\rho}\prod_{k=1}^{n-1}\Big(1+\frac{\rho}{2k}\Big).
	\end{equation}
\end{hypo}
One can express Conjecture \ref{Khbconj} in two other equivalent forms (\cite[Baladai--Khabibullin]{Kh10}), one of which is in terms of nonnegative increasing functions 
studied also by the author \cite{AB1, AB2, AB3}. As remarked by Khabibullin \cite{Kh02} the validation of this conjecture suffices to conclude that $P_n(\rho)$ is a sharp estimate for $\vartheta(u)$ when $\rho>1$.
However Sharipov \cite{Sharipov} constructs the first counterexample to this hypothesis. Following Sharipov the author extended this negative result in \cite{AB2} for $n=2$ and any $\rho>1$. In \cite{AB3} a counterexample was constructed for the general case $n\geq 2,\rho>1$. By Theorem \ref{Khthm} $\vartheta(u)$ is uniformly bounded by the constant $e^{n-1}P_n(\rho)$ when $\rho>1$. Therefore there must exists a best possible estimate $K_n(\rho)\leq e^{n-1}P_n(\rho)$ such that $\vartheta(u)\leq K_n(\rho)$ uniformly for all psh $u$ when $\rho>1$. We wish to call $K_n(\rho)$ the \textit{Khabibullin's constant}. Now note that Conjecture \ref{Khbconj} would be a true statement if for the second integral \eqref{eq:conj2} instead the bound $\displaystyle\frac{n-1}{2\rho^2} K_n(\rho)$ is  substituted. It is the main goal of this work to establish the exact formula for $K_n(\rho)$. More precisely we show that 
\begin{equation}
\label{eq:K-estimate}
K_n(\rho)= P_n(\rho)+\sum_{i\in I}\sum_{k=0}^n\Big(\Phi_{\rho, k}(\tau_{i-1})-\Phi_{\rho, k}(\tau_{i})\Big)
\end{equation}
where $\Phi_{\rho, k}$ are certain functions related to $\varphi_{\rho}(t)\coloneqq1/(1+t^{\rho})$ and $\tau_i, i\in I$ run over a certain subset of the zeros of the $n$-th derivative of $\varphi_{\rho}$.

\section{Preliminaries}
\subsection{Discussion}
For clarity and completeness we follow the steps of Khabibullin \cite{Kh99}.
Let $v$ be a psh function in $\mathbb{C}^n$ such that $v(0)=0$ and for some $\rho>0$ the inequality holds
  \begin{equation}
   \label{Nevestimate}
   T(r,v)\leq r^{\rho}\quad\forall r>0.
  \end{equation}
Theorem \ref{Khthm} follows from the estimate (see \cite[Main Lemma]{Kh99}) 
 \begin{align}
  \label{Mestimate} 
   M(1,v)\leq \left\{
       \begin{array}{lr}
        \displaystyle P_n(\rho) & : 0\leq\rho\leq\displaystyle 1\\
        e^{n-1}P_n(\rho) & : \rho>\displaystyle 1.
 \end{array}
     \right.
  \end{align} 
In view of \eqref{Mestimate} it is desirable to improve the upper bound on $M(1,v)$ when $\rho>1$.
Without loss of generality assume that $v$ is nonnegative.
For a fixed $\zeta\in\mathbb{S}^n$ consider the subharmonic (with respect to
$w\in\mathbb{C}$ ) slice function $v_{\zeta}(w)=v(\zeta w)$. By virtue of \cite[Lemma 1]{Kh99} the following estimate is true
\begin{equation}
 \label{vestimate}
 v_{\zeta}(1)\leq 4\mu^2\int^{+\infty}_0T(t,v_{\zeta})\varphi_{\mu}(t)\,dt\quad\text{where}\quad \mu=\max\Big\{\frac{1}{2},\rho\Big\},\quad\varphi_{\mu}(t)=\frac{t^{2\mu-1}}{(1+t^{2\mu})^2}
\end{equation}
and 
\begin{equation}
 \label{Nev}
 T(t,v_{\zeta})=\frac{1}{2\pi}\int^{2\pi}_0v_{\zeta}(te^{i\theta})\,d\theta
\end{equation}
is the Nevanlinna characteristic of the function $v_{\zeta}$. Inequality \eqref{vestimate} is a direct application of a more general result about subharmonic functions $u$ in $\mathbb{C}$ for which their Laplacian $\Delta u$ vanishes identically in a neighbourhood of the origin (see \cite[Lemma 4.1]{Essen}).
Given an arbitrary $k$-dimensional complex subspace $L_k\subset\mathbb{C}^n$ we let 
$T(r,v;L_k)$ be the Nevanlinna
characteristic of the restriction of $v$ to $L_k$.
$T(r,v;L_k)$ is nondecreasing and convex in $\log r$, therefore psh for each $k=1,2,...,n$. 
By \cite[Lemma 2]{Kh99} for $L_k\subset L_{k+1}$
\begin{equation}
 \label{relationL}
 r\int^r_0T(t,v;L_k)t^{2k-1}\,dt\leq \frac{1}{2k}T(r,v;L_{k+1})r^{2(k+1)-1}\hspace{0.3cm}r>0.
\end{equation}
An iterative application of \eqref{relationL} for a chain of complex subspaces $L_1\subset L_2\subset...\subset L_k$ for $k\leq n$, yields
\begin{align}
 \label{iterativeL} 
 t\int^t_0t_{k-1}\int^{t_{k-1}}_0t_{k-2}...\int^{t_3}_0t_2\int^{t_1}_0t_1T(t_1,v;L_1)&\,dt_1\,dt_2...\,dt_{k-1}\\\nonumber&\leq \frac{1}{2^{k-1}(k-1)!}T(t,v;L_k)t^{2k-1}.
\end{align}
If $G$ denotes the operator defined by 
\begin{equation*}
 \label{opQ}
 G[f](t)\coloneqq\int^t_0uf(u)\,du
\end{equation*}
then we have the identity 
\begin{equation*}
 \label{Qidentity}
 G^{k}[f](t)=\frac{1}{2^{k-1}(k-1)!}\int^t_0(t^2-u^2)^{k-1}uf(u)\,du\quad k\geq 1.
\end{equation*}
We can write \eqref{iterativeL} equivalently as 
\begin{equation}
\label{equivL}
 G^{k-1}[T(\cdot,v;L_1)](t)\leq\frac{1}{2^{k-1}(k-1)!}T(t,v;L_k)t^{2k-2}.
\end{equation}
By definition $T(t,v;L_1)=T(t,v_{\zeta})$ and $T(t,v;L_n)=T(t,v)$. In the special case $k=n$ we obtain 
\begin{equation}
\label{equivL}
 G^{n-1}[T(\cdot,v_{\zeta})](t)\leq\frac{1}{2^{n-1}(n-1)!}T(t,v)t^{2n-2}.
\end{equation}
\begin{lemma}\cite[Proposition 5.1]{Shabat}
\label{inc}
 A real valued function $S(t)$, with $S(0)=0$, is increasing on $[0,+\infty)$ and convex with respect to $\log t$ if and only if there exists an increasing function $s(t)$ on $[0,+\infty)$ such that $S(t)$ can be represented as 
 \begin{equation}
 \label{logconvex}
 S(t)=\int^t_0\frac{s(x)}{x}\,dx \quad t\geq 0.
 \end{equation}  
\end{lemma}

\begin{proposition}
	\label{p:s}
	Let $v$ be a psh function of finite lower order $\rho>0$. Then there exists a nonnegative increasing function $s(\cdot,v_{\zeta}):[0,+\infty)\to[0,+\infty)$ such that  
	\begin{equation}
		\label{eq:growth}
	\int^1_0(1-u)^{n-1}\frac{s(tu,v_{\zeta})}{u}\,du\leq t^{\rho/2}	
	\end{equation}
	and 
\begin{equation}
\label{vestimate2}
v_{\zeta}(1)\leq 2\rho\int^{+\infty}_0\frac{s(t,v_{\zeta})}{t}\frac{1}{1+t^{\rho}}\,dt.
\end{equation} 
\end{proposition}
\begin{proof}
	Note that $T(t, v_{\zeta})$ is nondecreasing on $[0,+\infty)$ and convex in $\log t$. Then a direct application of Lemma \ref{inc} to $T(t, v_{\zeta})$ yields
	\begin{equation*}
	T(t,v_{\zeta})=\int^t_0\frac{s(x,v_{\zeta})}{x}\,dx\quad t\geq0
	\end{equation*}
	for some increasing function $s(\cdot,v_{\zeta})$. Using this representation we get 
	\begin{align*}
	G[T(\cdot,v_{\zeta})](t)=\int^t_0uT(u,v_{\zeta})\,du=\int^t_0u\Big(\int^u_0\frac{s(x,v_{\zeta})}{x}\,dx\Big)\,du=\frac{1}{2}\int^t_0(t^2-u^2)\frac{s(u,v_{\zeta})}{u}\,du.
	\end{align*}
	From \eqref{equivL} and a change of variables $ut\equiv u$ implies
	\begin{equation*}
	\int^1_0(1-u^2)^{n-1}\frac{s(tu,v_{\zeta})}{u}\,du\leq T(t,v).
	\end{equation*}
	Substituting $u\equiv u^2$ and $2s(t^2u^2,v_{\zeta})\equiv s(tu,v_{\zeta})$ gives the final form 
	\begin{equation*}
	\label{equivfinal}
	\int^1_0(1-u)^{n-1}\frac{s(tu,v_{\zeta})}{u}\,du\leq T(t^{1/2},v)\leq t^{\rho/2}	.
	\end{equation*}
	This proves inequality \eqref{eq:growth}.
	On the other hand an integration by parts 
	\begin{equation*}
	\int^{+\infty}_0T(t,v_{\zeta})\varphi_{\mu}(t)\,dt= \int^{+\infty}_0\Big(\int^t_0\frac{s(x,v_{\zeta})}{x}\,dx\Big)\varphi_{\mu}(t)\,dt=\frac{1}{2\mu}\int^{+\infty}_0\frac{s(t,v_{\zeta})}{t}\frac{1}{1+t^{2\mu}}\,dt
	\end{equation*}
	together with $2s(t^2,v_{\zeta})\equiv s(t,v_{\zeta})$
	imply for \eqref{vestimate} the equivalent estimate 
	\begin{equation*}
	 v_{\zeta}(1)\leq 2\mu\int^{+\infty}_0\frac{s(t,v_{\zeta})}{t}\frac{1}{1+t^{\mu}}\,dt.
	\end{equation*}
	Since $\rho>1$ then $\mu=\rho$. This proves inequality \eqref{vestimate2}.
\end{proof}
Let $\psh_{\rho}(\mathbb C^n)$ denote the set of nonnegative functions $v$ that are psh in $\mathbb C^n$ such that inequality $T(r,v)\leq t^{\rho}$ is satisfied. Let $\inc_{\rho}(\mathbb R_+)$ be the set of nonnegative and increasing functions on $\mathbb R_+$ satisfying integral inequality \eqref{eq:growth}. Note that  $$M(1,v)=\max\{v(z)\;:\;z\in\mathbb C^n, |z|=1 \}=\max\{v_{\zeta}(1)\;:\;\zeta\in\mathbb S^n\}$$ then by Proposition \ref{p:s} for any psh function $v$ with lower order $\rho>0$ it holds that
$$M(1,v)\leq 2\rho\sup_{v\in\psh_{\rho}(\mathbb C^n)}\int^{+\infty}_0\frac{s(t,v_{\zeta})}{t}\frac{1}{1+t^{\rho}}\,dt= 2\rho \sup_{s\in\inc_{\rho}(\mathbb R_+)}\int^{+\infty}_0\frac{s(t)}{t}\frac{1}{1+t^{\rho}}\,dt.$$
The last equality follows  from the correspondence stated in Lemma \ref{inc} and the fact that to each increasing function $T(r)$ that is convex in $\log r$ one can associate a psh function $v$ in $\mathbb C^n$ such that $T(r)=T(r,v)$ i.e. take $v(z)\coloneqq T(|z|)$.
For a given $s\in\inc_{\rho}(\mathbb R_+)$ denote the last integral by $J_{\rho}(s)$  and $J(\rho)\coloneqq\sup_{s\in\inc_{\rho}(\mathbb R_+)}J_{\rho}(s)$.

\section{Boundedness of the functional $J_{\rho}$}
\subsection{An upper estimate for $J_{\rho}$}
	Let $Q$ be the operator acting on a integrable function $f$ by the formula
	\begin{equation*}
	Q[f](t)\coloneqq\int^t_0f(x)\,dx\quad t\geq 0.
	\end{equation*}
	To operator $Q$ corresponds an operator $Q^{-1}$, the inverse operator of $Q$, which acts on a differentiable function $f$ by the formula 
	\begin{equation*}
	Q^{-1}[f](t)\coloneqq\frac{df(t)}{dt}.
	\end{equation*}
A repeated integration yields the well known Cauchy integral identity
\begin{equation*}
Q^k[f](t)=\frac{1}{(k-1)!}\int^t_0(t-x)^{k-1}f(x)\,dx\quad\forall k\in\mathbb{N}.
\end{equation*}
In view of this operator we have
\begin{equation}
\label{eq:equiv}
\int^1_0(1-u)^{n-1}\frac{s(tu)}{u}\,du=\frac{1}{t^{n-1}}\int^t_0(t-u)^{n-1}\frac{s(x)}{x}\,du=\frac{(n-1)!}{t^{n-1}}Q^n\Big[\frac{s(x)}{x}\Big](t).
\end{equation}
	For $\rho>1$ let $\varphi_{\rho}:\mathbb{R}_+\to\mathbb{R}_+$ be the function
	\begin{equation}
	\label{varphi}
	\varphi_{\rho}(t)\coloneqq\frac{1}{1+t^{\rho}}\quad t\geq0.
	\end{equation}

\begin{lemma}
\label{twoest}
For all $k=0,1,2,...$ the followings hold
\begin{align}
  & |(Q^{-1})^{k}[\varphi_{\rho}](t)|=O(t^{\rho-k})\quad\text{as}\quad t\to0\label{first}\\&
   |(Q^{-1})^k[\varphi_{\rho}](t)|=O(t^{-\rho-k})\quad\text{as}\quad t\to+\infty\label{second}
  \end{align}
\end{lemma}
\begin{proof}
 Relation \eqref{first} follows from the power series expansion $\varphi_{\rho}(t)=1-t^{\rho}+t^{2\rho}-...$ as $t\to0$. Similarly \eqref{second} follows from the expansion $\varphi_{\rho}(t)=t^{-\rho}-t^{-2\rho}+...$ as $t\to+\infty$.
\end{proof}

\begin{lemma}
	\label{l:Qineq}
	For any $s\in\inc_{\rho}(\mathbb R_+)$ we have the inequalities
	\begin{equation}
		\label{eq:Qineq}
		Q^{k+1}\Big[\frac{s(x)}{x}\Big](t)\leq \frac{1}{k!}t^{\rho/2+k}\quad k=0,1,2,...
	\end{equation}
\end{lemma}
\begin{proof}
	Follows immediately from \eqref{eq:equiv} and Proposition \ref{p:s}.
\end{proof}

\begin{lemma}
 \label{vanishing}
 For all $k=0,1,2,...$ the following limits hold
 \begin{equation}
  \label{vanishinglimits}
  \lim_{t\to 0,+\infty}|(Q^{-1})^k[\varphi_{\rho}](t)|\cdot Q^{k+1}\Big[\frac{s(x)}{x}\Big](t)=0
 \end{equation}
\end{lemma}
 \begin{proof}
By Lemma \ref{l:Qineq} it follows
$$Q^{k+1}\Big[\frac{s(x)}{x}\Big](t)=O(t^{\rho/2+k})$$
By Lemma \ref{twoest} 
\begin{align}
 \label{asymptotic} 
 |(Q^{-1})^k[\varphi_{\rho}](t)|\cdot Q^{k+1}\Big[\frac{s(x)}{x}\Big](t)= \left\{
      \begin{array}{lr}
       \displaystyle O(t^{3\rho/2}) & : \text{as}\quad t\to 0\\
       O(t^{-\rho/2}) & : \text{as}\quad t\to+\infty
\end{array}
    \right.
 \end{align} 
 The limits in \eqref{vanishinglimits} follow immediately for all $k=0,1,2,...$. 
\end{proof}
\begin{mydef}
	\label{d:psi}
	For $\rho>1$ let $\psi_{\rho}:\mathbb{R}\to \mathbb{R}$ be defined as  
	\begin{equation}
	\label{psi}
	\psi_{\rho}(t)\coloneqq(-1)^n(Q^{-1})^n[\varphi_{\rho}](t)\quad t\geq0.
	\end{equation}
	Furthermore let $D_+\coloneqq\{t\in \mathbb{R}_+: \psi_{\rho}(t)\geq 0\}$ and $D_-\coloneqq\{t\in \mathbb{R}_+: \psi_{\rho}(t)< 0\}$.
\end{mydef}

\begin{proposition}
	\label{p:D}
If $\rho>1$ then for any $n\geq 2$  $\psi_{\rho}(t)$ vanishes at most on a finite number of points. In particular $\psi_{\rho}(t)>0$ for all large enough $t>0$.
\end{proposition}
\begin{proof}
By \cite[Proposition 3.2]{AB3} the equation $(Q^{-1})^n[\varphi_{\rho}](\tau)=0$ has a solution for all $n\geq 2$ and moreover there are at most a finite number of them. If $\sign(\cdot)$ denotes the usual sign function taking value $1$ on positive numbers, $-1$ on negative numbers and $0$ otherwise, then the set $\Omega\subseteq\mathbb{R}$ on which $\sign(Q^{-1})^n[\varphi_{\rho}](t)=(-1)^{n+1}$ can be expressed as a finite union of open bounded intervals in $\mathbb{R}_+$. In particular $\sign(Q^{-1})^n[\varphi_{\rho}](t)=(-1)^{n}$ for all sufficiently large $t$. By Definition \ref{d:psi} it follows that $\psi_{\rho}$ also vanishes at most on a finite number of points in $\mathbb{R}_+$ and  $$\sign\psi_{\rho}(t)=(-1)^n\sign(Q^{-1})^n[\varphi_{\rho}](t)=(-1)^n\cdot(-1)^{n+1}=-1\quad \forall t\in\Omega.$$ In particular $D_-=\Omega$ and $\psi_{\rho}(t)>0$ for all large enough $t$.
\end{proof}

\begin{remark}
	Note that from the arguments in the proof of the last proposition it follows that $D_-$ is a bounded set, in particular the closure of $D_-$ is compact. 
\end{remark}

\begin{mydef}
	\label{d:Phi}
Let $0\leq \tau_1\leq \tau_2\leq...\leq\tau_m<+\infty$ be an enumeration of the zeros of $\psi_{\rho}$. Define the index set $I\coloneqq\{i\in\mathbb{N}: \psi_{\rho}(t)<0,\forall t\in(\tau_{i-1},\tau_i)\}$. Denote by
\begin{equation}
\label{Phi1}
\Phi_{\rho, k}(t)\coloneqq(-1)^{n+k}\frac{\Gamma(\rho/2+n)}{\Gamma(n)\Gamma(\rho/2+n-k)}t^{\rho/2+n-k-1}\varphi^{(n-k)}_{\rho}(t)\quad\text{for}\quad k=0,1,...,n-1
\end{equation}
and
\begin{equation}
\label{Phi2}
\Phi_{\rho,n}(t)\coloneqq\frac{\Gamma(\rho/2+n)}{\Gamma(n)\Gamma(\rho/2+1)}\arctan t^{\rho}.
\end{equation}
\end{mydef}

\begin{proposition}
	\label{p:J}
For any $s\in\inc_{\rho}(\mathbb R_+)$ the following inequality holds:
\begin{align}
\label{eq:J} 
J_{\rho}(s)\leq 
\displaystyle\frac{P_n(\rho)}{2\rho} + \sum_{i\in I}\sum_{k=0}^n\Big(\Phi_{\rho, k}(\tau_{i-1})-\Phi_{\rho, k}(\tau_{i})\Big).
\end{align} 
\end{proposition}
\begin{proof}
	Let $\rho>1$ then for any $s\in\inc_{\rho}(\mathbb R_+)$ by \eqref{psi} and Lemma \ref{vanishing} we have
	\begin{align*}
	J_{\rho}(s)=\int^{+\infty}_0\frac{s(t)}{t}\varphi_{\rho}(t)\,dt=\int_0^{+\infty}Q^n\Big[\frac{s(x)}{x}\Big](t)\psi_{\rho}(t)\,dt.
	\end{align*}
Since $s\in\inc_{\rho}(\mathbb R_+)$ by Lemma \ref{l:Qineq} we obtain
\begin{align*}
\int_0^{+\infty}Q^n\Big[\frac{s(x)}{x}\Big](t)\psi_{\rho}(t)\,dt&\leq \int_{D_+}Q^n\Big[\frac{s(x)}{x}\Big](t)\psi_{\rho}(t)\,dt\leq\frac{1}{(n-1)!}\int_{D_+}t^{\rho/2+n-1}\psi_{\rho}(t)\,dt. 
\end{align*}	
The last integral can be written as 
\begin{align*}
\frac{1}{(n-1)!}\int_{D_+}t^{\rho/2+n-1}\psi_{\rho}(t)\,dt=\frac{1}{(n-1)!}\int_0^{+\infty}t^{\rho/2+n-1}\psi_{\rho}(t)\,dt-\frac{1}{(n-1)!}\int_{D_-}t^{\rho/2+n-1}\psi_{\rho}(t)\,dt.
\end{align*}
An integration by parts and Lemma \ref{twoest} imply
\begin{align*}
\frac{1}{(n-1)!}\int_0^{+\infty}t^{\rho/2+n-1}\psi_{\rho}(t)\,dt&=\frac{\rho}{2}\prod^{n-1}_{k=1}\Big(1+\frac{\rho}{2k}\Big)\int^{+\infty}_0\frac{t^{\rho/2-1}}{1+t^{\rho}}\,dt\\&=\prod^{n-1}_{k=1}\Big(1+\frac{\rho}{2k}\Big)\int^{+\infty}_0\frac{ dt^{\rho/2}}{1+t^{\rho}}=\frac{\pi}{2}\prod^{n-1}_{k=1}\Big(1+\frac{\rho}{2k}\Big)=\frac{P_n(\rho)}{2\rho}.
\end{align*}
Successive integration by parts implies the following identity in indefinite form
$$\int t^{\rho/2+n-1}\psi_{\rho}(t)\,dt=\frac{\Gamma(\rho/2+n)}{\Gamma(\rho/2+1)}\arctan t^{\rho}+\sum_{k=0}^{n-1}(-1)^{n+k}\frac{\Gamma(\rho/2+n)}{\Gamma(\rho/2+n-k)}t^{\rho/2+n-k-1}\varphi^{(n-k)}_{\rho}(t).$$
Then using \eqref{Phi1} and \eqref{Phi2} we can write 
$$\frac{1}{(n-1)!}\int_{D_-}t^{\rho/2+n-1}\psi_{\rho}(t)\,dt= \sum_{i\in I}\sum_{k=0}^n\Big(\Phi_{\rho, k}(\tau_i)-\Phi_{\rho, k}(\tau_{i-1})\Big).$$
Altogether we have for $\rho>1$ that 
$$J_{\rho}(s)\leq \frac{1}{(n-1)!}\int_{D_+}t^{\rho/2+n-1}\psi_{\rho}(t)\,dt=\frac{P_n(\rho)}{2\rho} + \sum_{i\in I}\sum_{k=0}^n\Big(\Phi_{\rho, k}(\tau_{i-1})-\Phi_{\rho, k}(\tau_{i})\Big).$$
This completes the proof.
\end{proof}

\begin{remark}
	\label{r:compact}
	By a theorem of the author \cite[Theorem 1]{AB3} there exists $s\in\inc_{\rho}(\mathbb R^+)$ such that $J_{\rho}(s)>P_n(\rho)/2\rho$. Therefore it follows that 
	$$\sum_{i\in I}\sum_{k=0}^n\Big(\Phi_{\rho, k}(\tau_{i-1})-\Phi_{\rho, k}(\tau_{i})\Big)>0.$$
\end{remark}

\begin{proposition}
\label{p:notatt} 
There does not exist $s\in \inc_{\rho}(\mathbb R_+)$ satisfying $J_{\rho}(s)=J(\rho)$ whenever $\rho>1$.
 \end{proposition}

 \begin{proof}
 Suppose that there is some $s\in \inc_{\rho}(\mathbb R_+)$ with $J_{\rho}(s)=J(\rho)$. Since $s(x)$ is increasing then  
by Lebesgue's differentiation theorem \cite[Theorem 9.3.1]{Hasser} $s$ is differentiable almost everywhere on $\mathbb{R}_+$.
Let $W$ be the set of points for which $s$ is not differentiable.
Define the function $$f(t)=Q^n\Big[\frac{s(x)}{x}\Big](t),\;t\geq 0.$$ By Proposition \ref{p:D} for any $n\geq 2$ the set $D_-\neq\emptyset$ whenever $\rho>1$ and that $\psi_{\rho}(t)$ vanishes at most on a finite number of points $\tau\in \mathbb{R}_+$. 
For $\varepsilon\in(0,1)$ let
$f_0$ be defined piecewise as
\begin{align}
\label{pw}  
f_0(t) \coloneqq \left\{
     \begin{array}{lr}
       f(t)(1- \varepsilon\displaystyle\eta(t)) & : t\in D_-\\
        f(t) & : t\in D_+
      \end{array}
    \right.
 \end{align} 
 where $\eta:\mathbb{R}_+\to\mathbb{R}_+$ is given by 
  \begin{align}
 \label{torus1}
 \eta(t)=\sum_{i\in I}\eta_i(t)\qquad\text{and}\qquad\eta_i(t)=\left\{\begin{array}{lr}\Big[\cos\Big(\displaystyle\frac{\pi t}{2\tau_{i-1}}\Big)\cos\Big(\displaystyle\frac{\pi t}{2\tau_i}\Big)\Big]^{2n}& :t\in(\tau_{i-1},\tau_i]\\
 0& : t\notin(\tau_{i-1},\tau_i] 
 \end{array}\right. 
 \end{align}
 with the agreement that if $1\in I$ then $\eta_1(t)=\cos^{2n}\Big(\displaystyle\frac{\pi t}{2\tau_1}\Big)$ for all $t\in(0,\tau_1]$ and zero otherwise. Here the index set $I$ and the finite sequence $(\tau_i)$ is given by Definition \ref{d:Phi}.
 Note that $\eta$ is smooth, nonnegative with support in $D_-$ and uniformly bounded above by one.
 Moreover by definition of $\eta$ the following vanishing limits are fulfilled
\begin{equation}
 \label{vanish2}
 \lim_{t\to \tau_i^-}\frac{d^k}{dt^k}[f(t)\eta(t)]=0\qquad\text{and}\qquad \lim_{t\to \tau_{i-1}^+}\frac{d^k}{dt^k}[f(t)\eta(t)]=0
 \end{equation}
 for $k=1,...,n+1$ and for all $\tau_i\in\psi^{-1}_{\rho}(0)$. To $f_0$ we can associate $s_0$ defined piecewise
 \begin{align}
 \label{PW}
 s_0(t)=t(Q^{-1})^n[f_0](t)= \left\{
     \begin{array}{lr}
     s(t)- \displaystyle\varepsilon t(Q^{-1})^n\displaystyle [f(\cdot)\eta(\cdot)](t) & : t\in D_-\cap (\mathbb{R}_+\setminus W)\\
      s(t) & : t\in D_+\cap (\mathbb{R}_+\setminus W)
      \end{array}
    \right.
 \end{align}
 In view of \cite[Theorem 3.1]{AB3} it is possible to find an $\varepsilon\in(0,1)$ such that $s_0, s'_0\geq 0$ on $\mathbb{R}_+\setminus W$. And on $W$ we can arrange the values of $s_0$ so that it is nonnegative and preserves monotonicity on $\mathbb{R}_+$. Hence $s_0$ is a nonnegative increasing function.
By \eqref{pw} and Lemma \ref{l:Qineq} we have $$Q^n\Big[\frac{s_0(x)}{x}\Big](t)=f_0(t)\leq \frac{t^{\rho/2+n-1}}{(n-1)!}$$ 
 therefore $s_0\in \inc_{\rho}(\mathbb R_+)$. On the other hand since $W$ is at most countable, hence of Lebesgue measure zero, then $W$ does not contribute to the integral $J_{\rho}(s_0)$. This in turn implies
$$J_{\rho}(s_0)=\int_{\mathbb{R}_+\setminus W}\frac{s_0(t)}{t}\varphi_{\rho}(t)\,dt=\int^{+\infty}_0\frac{s_0(t)}{t}\varphi_{\rho}(t)\,dt$$
which together with representation \eqref{PW} yields
$$J_{\rho}(s_0)=\int^{+\infty}_0\frac{s(t)}{t}\varphi_{\rho}(t)\,dt-\varepsilon\int_{D_-}(Q^{-1})^n\displaystyle [f(\cdot)\eta(\cdot)](t)\varphi_{\rho}(t)\,dt.$$
 By Proposition \ref{p:D} it follows that $D_-=\bigcup_{i\in I}(\tau_{i-1},\tau_i)$. So we can write $\int_{D_-}=\sum_{i\in I}\int^{\tau_i}_{\tau_{i-1}}$. Integrating by parts in each term of this sum and using the vanishing limits \eqref{vanish2} implies
 \begin{align*}
  J_{\rho}(s_0)=\int^{+\infty}_0\frac{s(t)}{t}\varphi_{\rho}(t)\,dt-\varepsilon\int_{D_-} f(t)&\eta(t)\psi_{\rho}(t)\,dt\\&>\int^{+\infty}_0\frac{s(t)}{t}\varphi_{\rho}(t)\,dt=J_{\rho}(s).
 \end{align*}
However this contradicts that $J_{\rho}(s)$ is maximal.
 \end{proof}

As an immediate corollary we obtain:

\begin{corollary}
	There is no $v\in\psh_{\rho}(\mathbb C^n)$, $\rho>1$ such that $M(1,v)=J(\rho)$.
\end{corollary}
\bigskip

\subsection{Constructing a maximizing sequence}
In view of Proposition \ref{p:notatt} it is then desirable to construct a sequence $(s_k)_{k\in\mathbb N}\subseteq\inc_{\rho}(\mathbb R_+)$ such that $\lim_{k\to+\infty}J_{\rho}(s_k)=J(\rho)$. The method of proof in the last proposition provides us with a concrete method of constructing such a maximizing sequence. Let $s_0\in\inc_{\rho}(\mathbb R_+)$ and denote by $f_0(t)\coloneqq Q^n[s_0(x)/x](t)$. By \cite[Theorem 3.1]{AB3}, for each $k\in\mathbb N$ we can find a positive constant $\varepsilon_k\in(0,1)$ such that 
the sequence of functions 

\begin{align}
\label{fk} 
f_k(t)\coloneqq \left\{
\begin{array}{lr}
f_{k-1}(t) & : t\in D_+\\
(1-\varepsilon_k\eta(t))f_{k-1}(t) & : t\in D_-.
\end{array}
\right.
\end{align} 
yields a corresponding sequence $(s_k)_{k\in\mathbb N}\subset\inc_{\rho}(\mathbb R_+)$ by the formula $s_k(t)\coloneqq t(Q^{-1})^n[f_k](t)$. Here $\eta$ is the function given by \eqref{torus1}. We start with 

\begin{equation}
	\label{eq:s0}
	s_0(t)\coloneqq\frac{\rho}{2}\,\prod_{k=1}^{n-1}\Big(1+\frac{\rho}{2k}\Big)\,t^{\rho/2},\quad t\geq 0.
\end{equation}
It is evident that $s_0\in\inc_{\rho}(\mathbb R_+)$. To $s_0$ corresponds 

\begin{equation}
	\label{eq:f0}
	f_0(t)=\frac{1}{(n-1)!}\,t^{\rho/2+n-1},\quad t\geq 0.
\end{equation}
By the recursion formula $f_k(t)\coloneqq(1-\varepsilon_k\eta(t))f_{k-1}(t)$ we obtain 
 
\begin{equation}
\label{eq:fk}
f_k(t)=\prod_{i=1}^{k}(1-\varepsilon_i\,\eta(t))\cdot f_0(t),\quad t\in D_-,\, k=1,2,3,...
\end{equation}
and correspondingly

\begin{equation}
\label{eq:sk}
s_k(t)=t\,(Q^{-1})^n\Big[\prod_{i=1}^{k}(1-\varepsilon_i\,\eta)\cdot f_0\Big](t),\quad t\in D_-,\, k=1,2,3,...
\end{equation}
Expanding the product $\displaystyle \prod_{i=1}^{k}(1-\varepsilon_i\,\eta)$ and using linearity of the operator $Q^{-1}$ we obtain 

\begin{align}
\label{eq:sk1}
s_k(t)=s_0(t)-t\,(Q^{-1})^n[\eta\,f_0](t)\sum_{i=1}^k\varepsilon_i+t\,&(Q^{-1})^n[\eta^2\,f_0](t)\,\sum_{i<j}\varepsilon_i\,\varepsilon_j+...\\\nonumber&...+(-1)^kt\,(Q^{-1})^n[\eta^k\,f_0](t)\,\prod_{i=1}^{k}\varepsilon_i.
\end{align}
Next we claim that for sufficiently large $k$ we can take $\varepsilon_k\sim \alpha/k$ for some constant $\alpha>0$. From \eqref{eq:sk1} the condition $s_k\geq 0$ is equivalent to the inequality 

\begin{align}
\label{eq:sk2}
1\geq\frac{t\,(Q^{-1})^n[\eta\,f_0](t)}{s_0(t)}\sum_{i=1}^k\varepsilon_i-&\frac{t\,(Q^{-1})^n[\eta^2\,f_0](t)}{s_0(t)}\,\sum_{i<j}\varepsilon_i\,\varepsilon_j+...\\\nonumber&...+(-1)^{k+1}\frac{t\,(Q^{-1})^n[\eta^k\,f_0](t)}{s_0(t)}\,\prod_{i=1}^{k}\varepsilon_i.
\end{align}
By definitions of $f_0$ and $s_0$ it follows that each ratio function in the last expression is continuous. As a consequence by Remark \ref{r:compact} the absolute value of each one attains a certain maximum on the closure of $D_-$. Additionally for sufficiently large $k$ the last term
$$\Big|\frac{t\,(Q^{-1})^n[\eta^k\,f_0](t)}{s_0(t)}\Big|\sim C\,\frac{k!}{(k-n)!}$$
dominates the whole sum. Here $C>0$ is a certain constant possibly depending on $f_0,s_0,\eta$ and $n$. Without loss of generality assume that $\varepsilon_k\leq\varepsilon_{k-1}\leq...\leq\varepsilon_1$ then for large $k$ 
$$1\geq C\, \frac{k!}{(k-n)!}\prod_{i=1}^{k}\varepsilon_i\geq C\, \frac{k!}{(k-n)!}\,(\varepsilon_k)^k$$
whenever $\varepsilon_k\sim \alpha/k$ for a certain constant $\alpha>0$. Hence such a choice for $\varepsilon_k$ is possible. In particular we get that $\sum_{k\geq k'}^N\varepsilon_k\to+\infty$ as $N\uparrow+\infty$ for $k'$ itself sufficiently large. Then $\sum_{k=1}^{N}\varepsilon_k\to+\infty$ as $N\uparrow+\infty$. By standard arguments from the theory of infinite series it follows that $\prod_{i=1}^{k}(1-\varepsilon_{k}\eta(t))\to 0$ as $k\uparrow +\infty$ for each $t\in D_-$. In particular $\lim_{k\to+\infty}f_k(t)=0$ for every $t\in D_-$. 
These observations help us show the following result:

\begin{proposition}
	\label{p:max}
	Let $s_0, f_0$ be given by \eqref{eq:s0} and \eqref{eq:f0} respectively. Then the sequence $(s_k)_{k\in\mathbb N}\subset\inc_{\rho}(\mathbb R_+)$ as constructed above is a maximizing sequence i.e. $\lim_{k\to+\infty}J_{\rho}(s_k)=J(\rho)$.
	In particular it holds that
	\begin{equation}
	\label{eq:max-seq}
		J(\rho)=\displaystyle\frac{P_n(\rho)}{2\rho} + \sum_{i\in I}\sum_{k=0}^n\Big(\Phi_{\rho, k}(\tau_{i-1})-\Phi_{\rho, k}(\tau_{i})\Big).
	\end{equation} 

\end{proposition}
\begin{proof}
For every $k\in\mathbb N$ we have 

\begin{align*}
J_{\rho}(s_k)=\int^{+\infty}_0\frac{s_k(t)}{t}\,\varphi_{\rho}(t)\,dt&=\int_{D_+}\frac{s_k(t)}{t}\,\varphi_{\rho}(t)\,dt+\int_{D_-}\frac{s_k(t)}{t}\,\varphi_{\rho}(t)\,dt\\&
=\int_{D_+}\frac{s_k(t)}{t}\,\varphi_{\rho}(t)\,dt+\int_{D_-}(Q^{-1})^n[f_k](t)\,\varphi_{\rho}(t)\,dt\\&
=\int_{D_+}\frac{s_k(t)}{t}\,\varphi_{\rho}(t)\,dt+\int_{D_-}f_k(t)\,\psi_{\rho}(t)\,dt.
\end{align*}
On the other hand by construction $f_0\geq f_1\geq f_2\geq...$, implying in particular that 
$$\int_{D_-}f_k(t)\,|\psi_{\rho}(t)|\,dt\leq \int_{D_-}f_0(t)\,|\psi_{\rho}(t)|\,dt<\int_{0}^{+\infty}f_0(t)\,|\psi_{\rho}(t)|\,dt<+\infty.$$
By Lebesgue's monotone convergence theorem \cite[Corollary 6.8.2]{Hasser} it follows that 
$$\lim_{k\to+\infty}\int_{D_-}f_k(t)\,\psi_{\rho}(t)\,dt=\int_{D_-}\lim_{k\to+\infty}f_k(t)\,\psi_{\rho}(t)\,dt=0.$$
Notice that by construction $s_k=s_0$ on $D_+$, therefore we get 
$$\lim_{k\to+\infty}J_{\rho}(s_k)=\lim_{k\to+\infty}\int_{D_+}\frac{s_k(t)}{t}\,\varphi_{\rho}(t)\,dt+\int_{D_-}f_k(t)\,\psi_{\rho}(t)\,dt=\int_{D_+}\frac{s_0(t)}{t}\,\varphi_{\rho}(t)\,dt.$$
The last integral was evaluated in Proposition \ref{p:J} and equals the estimate in \eqref{eq:J}. Because $\lim_{k\to+\infty}J_{\rho}(s_k)\leq J(\rho)$ then by Proposition \ref{p:J} the equality \eqref{eq:max-seq} follows. 

\end{proof}

\begin{remark}
	\label{r:3}
By \eqref{Mestimate} we have the estimate $M(1,v)\leq e^{n-1}P_n(\rho)$ for every $n\geq 2$ and $\rho>1$. On the other hand by Proposition \ref{p:max} we obtain $M(1,v)\leq 2\rho\,J(\rho)$ where $J(\rho)$ is given by the formula \eqref{eq:max-seq}. Since the later is achieved as the limit of a maximizing sequence then $2\rho \,J(\rho)\leq e^{n-1}P_n(\rho)$ must hold i.e. we have an improvement of the estimate when $\rho>1$. 
\end{remark}

\section{Main Results}

\subsection{P\'olya peaks} A main ingredient in the proof of the main result is the utilization of the so called P\'olya peaks.
Let $\rho>0$. A sequence $r_1,r_2,r_3,...$ is said to be a sequence of P\'olya peaks for a function $T(r)$ if there exist sequences of positive numbers $(\varepsilon_{k}), (\xi_k), (a_k), (A_k)$ satisfying
\begin{equation}
\label{eq:PP-ineq}
\lim_{k\to+\infty}\varepsilon_{k}=\lim_{k\to+\infty}\xi_k=\lim_{k\to+\infty}a_k=0\quad\text{and}\quad \lim_{k\to+\infty}a_kr_k=\lim_{k\to+\infty}A_k=+\infty
\end{equation}
such that the inequalities $r_ka_k\leq r\leq r_kA_k$ imply 
\begin{equation}
\label{eq:PP-ineq2}
T(r)\leq(1+\xi_k)\Big(\frac{r}{r_k}\Big)^{\rho+\varepsilon_{k}}T(r_k).
\end{equation}
Edrei \cite{Edrei1} proved that a sequence of P\'olya peaks for $T(r)$ always exists whenever $T(r), r\geq 1$ is unbounded, nondecreasing, nonnegative continuous function satisfying
\begin{equation}
\label{eq:Polya-peaks}
\liminf_{r\to+\infty}\frac{\log T(r)}{\log r}=\rho<+\infty.
\end{equation}

For a nonnegative psh function $u$ of lower order $\rho$ its Nevalinna characteristic $T(r,u)$ (possibly unbounded) is nondecreasing, nonnegative continuous function of $r\geq 0$. In view of definition \eqref{lowerorder} as an immediate corollary we have:

\begin{corollary}
	\label{c:Polya}
	Let $u$ be a nonnegative psh function in $\mathbb C^n$ of lower order $\rho$, then $T(r,u)$ has a sequence of P\'olya peaks of order $\rho$. 
\end{corollary}

\begin{lemma}
	\label{l:uniformly-bounded}
	Let $u$ be a nonnegative psh function in $\mathbb C^n$ of lower order $\rho$ and let $(r_k)$ be the sequence of P\'olya peaks for $T(r,u)$. Define 
	\begin{equation}
	\label{eq:uk}
	 u_k(z)\coloneqq\frac{u(r_kz)}{T(r_k,u)}.
	\end{equation}
	Then $(u_k)$ is a sequence of psh functions in $\mathbb C^n$ that is uniformly bounded on compact subsets of $\mathbb C^n$.
\end{lemma}

\begin{proof}
	First note that for every $k\in\mathbb N$ the function $u_k$ is psh in $\mathbb C^n$ since $T(r_k,u)>0$ and $u(r_kz)$ is psh in $\mathbb C^n$. For a given nonnegative function $v$ subharmonic in $\mathbb C^n$ the following inequality holds true $M(r,v)\leq 3\cdot 2^{2n-2}\,T(2r,v)$ (see \cite[(2.2)]{Kh99}). Because a psh function is subharmonic then applying this to each $u_k$ yields 
	$$0\leq u_k(z)=\frac{u(r_kz)}{T(r_k,u)}\leq \frac{M(r_k\,r,u)}{T(r_k,u)}\leq 3\cdot 2^{2n-2}\frac{T(2r_k\,r,u)}{T(r_k,u)} ,\quad |z|\leq r.$$
	In view of inequality \eqref{eq:PP-ineq2} for $a_k\leq 2r\leq A_k$ we obtain 
	$$\frac{T(2r_k\,r,u)}{T(r_k,u)}\leq (1+\xi_k)r^{\rho+\varepsilon_{k}}\leq (1+\xi)\,r^{\rho+\varepsilon}$$
	for some $\xi, \varepsilon>0$. Consequently for $a_k\leq 2|z|\leq A_k$ we get 
	$$0\leq u_k(z)\leq 3\cdot 2^{2n-2}(1+\xi)\,r^{\rho+\varepsilon}.$$ 
	When $2|z|\leq a_k$ then again from definition of $u_k$ it follows that 
	$$0\leq u_k(z)\leq  \frac{M(a_k\,r_k/2,u)}{T(r_k,u)}\leq \frac{M(\,r_k/2,u)}{T(r_k,u)}\leq 3\cdot 2^{2n-2}\quad \text{for sufficiently large}\,k.$$
	In both situations $u_k$ is uniformly bounded on the ball $\mathbb B_k:=\{z\in\mathbb C^n\;:\;|z|\leq A_k/2\}$  for sufficiently large $k$. Now note that for any compact set $K\subset \mathbb C^n$ we have $K\subseteq\mathbb B_k$ for sufficiently large $k$. This completes the proof.
\end{proof}

An important part in the proof of the theorem plays the following result about sequences of subharmonic functions that are uniformly bounded on compact sets. While its original statement \cite[Theorem 4.1.9]{Hormander} is for subharmonic function in $\mathbb R^m$, the arguments extend equally also for subharmonic functions in $\mathbb C^n$, by identifying $\mathbb C^n\cong\mathbb R^{2n}$. 

\begin{lemma}\cite[Theorem 4.1.9]{Hormander}
	\label{l:Hormander}
Let $L^1_{\loc}(\mathbb C^n)$ denote the space of locally integrable functions in $\mathbb C^n$. and let $(v_k)$ be a sequence of subharmonic functions in $\mathbb C^n$ that is uniformly bounded on compact subsets of $\mathbb C^n$. If $v_k$ does not converge to $-\infty$ uniformly  on every compact subset of $\mathbb C^n$ then there exists a subsequence $(v_{n_k})$ convergent to a subharmonic function $v$ in $L^1_{\loc}(\mathbb C^n)$. Moreover the following inequality holds
\begin{equation}
\label{eq:Hormander}
\limsup_{k\to+\infty}v_{n_k}(z)\leq v(z)\quad z\in\mathbb C^n
\end{equation} 
and the two sides equal and finite almost everywhere.
\end{lemma}

\subsection{Main Theorems}
\begin{thm}
\label{Maintheorem}
 Let $u:\mathbb{C}^n\to[-\infty,+\infty)$ be a psh function of finite lower order $\rho>1$ then $\vartheta(u)\leq 2\rho\,J(\rho)$ where $J(\rho)$ is given by \eqref{eq:max-seq}. In particular $K_n(\rho)=2\rho\,J(\rho)$.
\end{thm}
\begin{proof}
 We follow the lines in the proof of \cite[Theorem 1]{Kh99}. Let $u:\mathbb{C}^n\to[-\infty,+\infty)$ be a psh function of finite lower order $\rho>1$. Assume without loss of generality that $u\geq0$. If $T(r,u)$ is the Nevanlinna characteristic of $u$ then by Corollary \ref{c:Polya} it follows that $T(r,u)$ has a sequence of P\'olya peaks $(r_k)$ of order $\rho$. Let $(u_k)$ be defined as in \eqref{eq:uk}. By Lemma \ref{l:uniformly-bounded} the sequence $(u_k)$ is uniformly bounded on compact subsets of $\mathbb C^n$. Then by Lemma \ref{l:Hormander} there is a subsequence $(u_{n_k})$ which converges to a nonnegative subharmonic function $v$ in $L^1_{\loc}(\mathbb{C}^n)$. Moreover $\limsup_{k\to+\infty}u_{n_k}(z)\leq v(z)$ for all $z\in\mathbb{C}^n$. If $\limsup_{k\to+\infty}u_{n_k}(z)\coloneqq w(z)$ we let $w^*(z)\coloneqq\limsup_{y\to z}w(y)$ denote the upper semicontinuous regularization of $w(z)$. It can easily be checked that $w^*(z)$ is upper semicontinuous and $w^*$ is the least function with the property $w^*\geq w$. On the other hand $v$ is subharmonic, hence upper semicontinuous. We then have the inequalities
\begin{equation}
 \label{regularization}
 \limsup_{k\to+\infty}u_{n_k}(z)=w(z)\leq w^*(z)\leq v(z)=v^*(z).
\end{equation}
Convergence of $(u_{n_k})$ to $v$ in $L^1$ norm implies that $w(z)$ and $v(z)$ coincide almost everywhere in $\mathbb{C}^n$.  Indeed let $K\subset\mathbb C^n$ be a compact set. From the triangle inequality we have 
$$0\leq\int_K|w(z)-v(z)|\,d\lambda(z)\leq \int_K|w(z)-u_{n_k}(z)|\,d\lambda(z)+\int_K|u_{n_k}(z)-v(z)|\,d\lambda(z)$$
where $\,d\lambda$ is the usual (real) Lebesgue measure in $\mathbb C^n$.
Passing in the limit we obtain
\begin{align*}
\int_K|w(z)-v(z)|\,d\lambda&\leq\limsup_{k\to+\infty}\int_K|w(z)-u_{n_k}(z)|\,d\lambda+\lim_{k\to+\infty}\int_K|u_{n_k}(z)-v(z)|\,d\lambda\\&
=\limsup_{k\to+\infty}\int_K|w(z)-u_{n_k}(z)|\,d\lambda
\leq \int_K\limsup_{k\to+\infty}|w(z)-u_{n_k}(z)|\,d\lambda=0
\end{align*}
where the passage of the limit inside the integral is justified by the (reversed) Fatou's Lemma \cite[Lemma 6.8.5]{Hasser} since $|w(z)-u_{n_k}(z)|\leq |w(z)|+C_K\leq 2C_K$ for every $k\in\mathbb  N$ and a certain constant $C_K$ possibly dependent on $K$. Therefore $w(z)=v(z)$ almost everywhere on $K$. Since $K\subset\mathbb C^n$ is arbitrary then $w(z)=v(z)$ almost everywhere on $\mathbb C^n$. In particular inequality \eqref{regularization} yields $w^*(z)=v(z)$ almost everywhere.

Now consider any complex line $a+zb\subset\mathbb{C}^n$ where $a,b\in\mathbb{C}^n$ are fixed and $z\in\mathbb{C}$. By definition of plurisubharmonicity it follows $u_{n_k}(a+zb)$ is subharmonic in $z\in\mathbb{C}$. Denote  $w_m(z)\coloneqq\sup_{k\geq m}u_{n_k}(a+zb)$ for each $z\in \mathbb{C}$. Then $(w_m)$ is a nonincreasing sequence of functions and $\lim_mw_m(z)=w(a+zb)$ for each $z\in\mathbb{C}$. Let $w^*_m$ denote the upper semicontinuous regularization of $w_m$. By \cite[Theorem 3.2.2]{Hormander1} $w^*_m$ is subharmonic. On the other hand $w^*_m$ is the least function for which $w^*_m\geq w_m$ implying $\lim_mw^*_m(z)=w^*(a+zb)$. Because $(w^*_m)$ is a decreasing sequence of subharmonic functions then the limit $w^*(a+zb)$ is subharmonic in $z\in\mathbb{C}$. Since the complex line $a+zb\subset\mathbb{C}^n$ is arbitrary then by definition $w^*$ is psh in $\mathbb{C}^n$.  On the other hand two subharmonic functions that coincide almost everywhere are the same. 
Therefore $v$ is psh function in $\mathbb{C}^n$. By \cite[Theorem 4]{Kondratyuk} the sequence $(u_{n_k})$ can be chosen so that 
$$\lim_{k\to+\infty}\int_{\mathbb S_n}|u_{n_k}(r\,\zeta)-v(r\,\zeta)|\,ds_n(\zeta)=0,\quad r>0.$$
This in turn implies 
\begin{align*}
T(r,v)=\int_{\mathbb S_n}v(r\,\zeta)\,ds_n(\zeta)&\leq \int_{\mathbb S_n}u_{n_k}(r\,\zeta)\,ds_n(\zeta)+o(1)\\&
=\int_{\mathbb S_n}\frac{u(r_{n_k}r\,\zeta)}{T(r_{n_k},u)}\,ds_n(\zeta)+o(1)=\frac{T(r_{n_k}r,u)}{T(r_{n_k},u)}+o(1),\quad\text{as}\;k\uparrow+\infty.
\end{align*}
By \eqref{eq:PP-ineq2} it follows $T(r,v)\leq r^{\rho}$ and thus $v\in\psh_{\rho}(\mathbb C^n)$.
Moreover by means of the Poisson kernel (see \cite[Theorem 1]{Kh99}) it is shown that 
\begin{equation}
 \label{Max}
 \limsup_{k\to+\infty}M(1,u_k)\leq M(1,v).
\end{equation}
Definition of $(u_k)$ implies
\begin{align*}
 M(1,u_k)=M\Big(1,\frac{u(r_kz)}{T(r_k,u)}\Big)&=\max\Big\{\frac{u(r_kz)}{T(r_k,u)}\;:\;|z|=1,z\in\mathbb{C}^n\Big\}\\&=\frac{1}{T(r_k,u)}\max\{u(z)\;:\;|z|=r_k,z\in\mathbb{C}^n\}=\frac{M(r_k,u)}{T(r_k,u)}.
\end{align*}
Hence by \eqref{Max}
\begin{equation}
 \label{limitk}
 \vartheta(u)=\liminf_{r\to+\infty}\frac{M(r,u)}{T(r,u)}\leq \limsup_{k\to+\infty}\frac{M(r_k,u)}{T(r_k,u)}\leq M(1,v).
\end{equation}
On the other hand  $M(1,v)\leq 2\rho J(\rho)$. By Proposition \ref{p:J} it follows that 
$$\vartheta(u)\leq P_n(\rho) + 2\rho\sum_{i\in I}\sum_{k=0}^n\Big(\Phi_{\rho, k}(\tau_{i-1})-\Phi_{\rho, k}(\tau_{i})\Big).$$
Now we demonstrate that this estimate for $\vartheta(u)$ is nonimprovable i.e. $K_n(\rho)=2\rho\,J(\rho)$. It suffices to show that there is some psh function $u_0$ in $\mathbb C^n$ of finite lower order $\rho>1$ such that all inequalities in \eqref{limitk} hold with equality. To this end for $\rho>1$ let 
\begin{equation}
\label{eq:Mittag-Leffler}
E_{\rho}(z):=\sum_{j=0}^{+\infty}\frac{z^j}{\Gamma(1+j/\rho)},\quad z\in\mathbb C.
\end{equation}
  Define $u_0(z):=\log|E_{\rho}(z_1)|$ where $z:=(z_1,...,z_n)\in\mathbb C^n$. Since $E_{\rho}$ is an entire function of order $\rho$ it can be shown by direct calculations that $u_0$ is a psh function in $\mathbb C^n$ of finite lower order $\rho$. The following asymptotic for the characteristic functions of $u_0$ hold
 \begin{align}
 \label{eq:asymptotics}
  M(r,u_0)=r^{\rho}+O(1)\quad \text{and}\quad T(r,u_0)=P^{-1}_n(\rho)\,r^{\rho}+O(1).
 \end{align}
 This implies that  $$\vartheta(u_0)=\liminf_{r\to+\infty}\frac{M(r,u_0)}{T(r,u_0)}=\lim_{r\to+\infty}\frac{M(r,u_0)}{T(r,u_0)}=P_n(\rho)$$ and in particular for any P\'olya peaks $(r_k)$ of $T(r,u_0)$ it holds that 
 $$\limsup_{k\to+\infty}\frac{M(r_k,u_0)}{T(r_k,u_0)}=\vartheta(u_0).$$ 
 By construction of the sequence $(u_{0,k})$ from a given P\'olya peaks $(r_k)$ we have 
 $$u_{0,k}(z)=\frac{u_0(r_kz)}{T(r_k,u_0)}\leq\frac{M(r_kr,u_0)}{T(r_k,u_0)},\quad z\in\mathbb C^n\;\text{with}\;|z|=r.$$
 By the asymptotic formulae \eqref{eq:asymptotics} we obtain that 
 $$u_{0,k}(z)\leq \frac{r_k^{\rho}\,r^{\rho}+O(1)}{P^{-1}_n(\rho)\,r_k^{\rho}+O(1)}$$
 and consequently $\limsup_{k\to+\infty}u_{0,k}(z)\leq P_n(\rho)\,r^{\rho}$ for every $z\in\mathbb C^n$ with $|z|=r$.
 From the earlier arguments we have shown that $\limsup_{y\to z}\limsup_{k\to+\infty}u_{0,k}(y)= v_0(z)$ for some psh function $v_0\in\psh_{\rho}(\mathbb C^n)$. In particular we obtain that $v_0(z)\leq P_n(\rho)\,|z|^{\rho}$ and consequently $M(1,v_0)\leq P_n(\rho)\, M(1,|z|^{\rho})=P_n(\rho)$. Therefore $\vartheta(u_0)=M(1,v_0)$. 
\end{proof}

For a plurisubharmonic function $u$ of finite lower order $\rho$ let $\rho(r)$ denote its refined order i.e. $\lim_{r\to+\infty}\rho(r)=\rho$ and $\lim_{r\to+\infty}r\rho'(r)\ln r=0$ (see \cite[B. Ja. Levin, p.32]{Levin}). Let $\sigma_T$ and $\sigma_M$ be types of the characteristic functions $T(r,u)$ and $M(r,u)$ respectively with respect to the refined order $\rho(r)$. 
Then we obtain the following comparison result.

\begin{thm}
 The following estimates (nonimprovable) hold true
 \begin{align}
 \label{type} 
 \sigma_M\leq \left\{
      \begin{array}{lr}
       \displaystyle P_n(\rho)\,\sigma_T & : 0\leq\rho\leq\displaystyle 1\\
     K_n(\rho)\,\sigma_T & : \rho>\displaystyle 1.
\end{array}
    \right.
 \end{align}
\end{thm}

\begin{proof}
 By definition of the type $\sigma$ of a given function $f$ with respect $\rho(r)$ we have 
 $$\sigma_M\coloneqq\limsup_{r\to+\infty}\frac{M(r,u)}{r^{\rho(r)}}\quad\text{and}\quad \sigma_T\coloneqq\limsup_{r\to+\infty}\frac{T(r,u)}{r^{\rho(r)}}.$$
 By Theorem \ref{Maintheorem} for sufficiently small $\varepsilon>0$ and sufficiently large $k$ the following inequality 
 \begin{align}
 \label{approx}
 \frac{M(r_k,u)}{T(r_k,u)}\leq \left\{
      \begin{array}{lr}
       \displaystyle P_n(\rho)+\varepsilon & : 0\leq\rho\leq\displaystyle 1\\
        K_n(\rho)+\varepsilon & : \rho>\displaystyle 1
\end{array}
    \right.
 \end{align}
 holds true where $(r_k)$ is the P\'olya peaks of $T(r,u)$. Writing
 $$\frac{M(r_k,u)}{T(r_k,u)}=\frac{M(r_k,u)}{r_k^{\rho(r_k)}}\cdot\frac{r_k^{\rho(r_k)}}{T(r_k,u)}$$
 and passing in the limit superior in \eqref{approx} yield
 \begin{align*}
 \limsup_{k\to+\infty}\frac{M(r_k,u)}{r_k^{\rho(r_k)}}\leq \left\{
      \begin{array}{lr}
       \displaystyle P_n(\rho)\limsup_{k\to+\infty}\frac{T(r_k,u)}{r_k^{\rho(r_k)}} & : 0\leq\rho\leq\displaystyle 1\\
       K_n(\rho)\displaystyle\limsup_{k\to+\infty}\frac{T(r_k,u)}{r_k^{\rho(r_k)}} & : \rho>\displaystyle 1.
\end{array}
    \right.
 \end{align*}
 On the other hand 
 $$\limsup_{k\to+\infty}\frac{T(r_k,u)}{r_k^{\rho(r_k)}}=\limsup_{r\to+\infty}\frac{T(r,u)}{r^{\rho(r)}}=\sigma_T\quad\text{and}\quad\limsup_{k\to+\infty}\frac{M(r_k,u)}{r_k^{\rho(r_k)}}=\limsup_{r\to+\infty}\frac{M(r,u)}{r^{\rho(r)}}=\sigma_M.$$
\end{proof}

\subsection{Calculations for $n=2$}
For small values of $n$ it is possible to get explicit expression for Khabibullin's constant $K_{n}(\rho)$. In particular when $n=2$ routine calculations show that 
$$\psi_{\rho}(t)=\frac{\rho t^{\rho-2}((\rho+1)t^{\rho}-(\rho-1))}{(1+t^{\rho})^3}$$
and evidently $\psi_{\rho}(t)<0$ for $t\in[0,\tau)$ and $\psi_{\rho}(t)>0$ for $t>\tau$ where $\tau$ is the non-zero solution of $\psi_{\rho}(t)=0$ given by the formula
\begin{equation*}
\tau(\rho)=\Big(\frac{\rho-1}{\rho+1}\Big)^{1/\rho}.
\end{equation*}
In this case $D_-=(0,\tau)$ and from the general formula for the indefinite integral we obtain 
$$\int_0^{\tau}t^{\rho/2+1}\psi_{\rho}(t)\,dt=-\frac{(\rho+1)^2}{4\rho}\Big(\frac{\rho-1}{\rho+1}\Big)^{3/2}-\Big(\frac{\rho}{2}+1\Big)\Big[\frac{\rho+1}{2\rho}\Big(\frac{\rho-1}{\rho+1}\Big)^{1/2}-\arctan{\frac{\rho-1}{\rho+1}}\Big]$$
whenever $\rho>1$. Hence for $n=2$ and $\rho>1$ we have 
\begin{equation}
\label{eq:K-2}
K_2(\rho)=\Big(\frac{\rho}{2}+1\Big)\frac{\pi}{2}+\frac{(\rho+1)^2}{4\rho}\Big(\frac{\rho-1}{\rho+1}\Big)^{3/2}+\Big(\frac{\rho}{2}+1\Big)\Big[\frac{\rho+1}{2\rho}\Big(\frac{\rho-1}{\rho+1}\Big)^{1/2}-\arctan{\frac{\rho-1}{\rho+1}}\Big].
\end{equation}
Note that $\lim_{\rho\downarrow 1}K_2(\rho)=3\pi/4$ coincides with the sharp estimate when $\rho= 1$. Hence $K_2(\rho)$ is continuous at $\rho=1$. Moreover we get an asymptotic behaviour of $K_2(\rho)$ given by
$$\lim_{\rho\uparrow+\infty}\frac{K_2(\rho)}{\rho}= \frac{4+\pi}{8}.$$
Again from the asymptotic analysis we obtain 
$$K_2(\rho)= \Big(\frac{\rho}{2}+1\Big)\frac{\pi}{2}+o(\rho)\;\text{as}\;\rho\downarrow 1$$
and 
$$K_2(\rho)= \Big(\frac{\rho}{2}+1\Big)\frac{\pi}{2}+\frac{\rho}{4}+\Big(\frac{\rho}{2}+1\Big)\,\Big(\frac{1}{2}-\frac{\pi}{4}\Big)<2\,\Big(\frac{\rho}{2}+1\Big)\frac{\pi}{2}-\frac{1}{2}\;\text{as}\;\rho\uparrow +\infty.$$
In both case as $\rho$ approaches $1$ from above or as $\rho$ gets arbitrarily large the estimate shows once more that it is smaller then $e\,P_2(\rho)$.
Calculating $K_{n}(\rho)$ explicitly when $n\geq 3$  gets complicated because of the form of $\psi_{\rho}(t)$. It is known that 
$$\psi_{\rho}(t)=(-1)^n\frac{p(t^{\rho})}{t^n(1+t^{\rho})^{n+1}}\quad\text{for}\quad t\geq 0,  n\geq 2$$
where $p$ is a polynomial in $t^{\rho}$ of the form
$p(x)=c_{n,n-1}x^{n}+c_{n,n-2}x^{n-1}+...+c_{n,1}x^2+c_{n,0}x$.
The coefficients $c_{n,k}$ are polynomials in $\rho$ satisfying a certain recurrence relation (see \cite[Proposition 3.1]{AB3}) which can prove to be helpful in studying the zeros of $p(x)$. These in turn determine the zeros $\tau$ of the function $\psi_{\rho}$.

\subsection{Comparing with other estimates for subharmonic functions}
Because the set of psh functions in $\mathbb C^n$ is a proper subset of the more general class of subharmonic functions in $\mathbb C^n$, it is of interest to compare the estimate $K_n(\rho)$ for psh functions of lower order $\rho> 1$ with the estimate for subharmonic functions of the same lower order. But $K_n(\rho)$ is sharp for the class of psh functions, therefore it must be the case that $K_n(\rho)$ is no greater than any nonimprovable estimate for subharmonic functions. Indeed we verify that this is the case by making explicit comparison with the estimate obtained by Dahlberg \cite{Dahlberg} for subharmonic functions in $\mathbb R^m, m\geq3$. 
For $\rho>0$ let the Gegenbauer functions $C_{\rho}^{\gamma}$ be given as the solutions of the differential equation 
\begin{equation}
\label{eq:Gegenbauer-diff}
(1-x^2)\,\frac{d^2f}{dx^2}-(2\gamma+1)x\,\frac{df}{dx}+\rho(\rho+2\gamma)f=0,\quad -1<x<1
\end{equation}
with the normalization condition 
\begin{equation}
	\label{eq:normalization}
	\lim_{x\uparrow 1}C^{\gamma}_{\rho}(x)=C_{\rho}^{\gamma}(1)=\frac{\Gamma(\rho+2\gamma)}{\Gamma(2\gamma)\,\Gamma(\rho+1)}.
\end{equation}
Let $a_{\rho}\coloneqq\sup\{t \;:\;C^{\frac{m-2}{2}}_{\rho}(t)=0\}$ and define $u_{\rho}$ in $\mathbb R^m, m\geq 3$ by the formula:

\begin{align}
\label{u-rho} 
u_{\rho}(x)\coloneqq\left\{
\begin{array}{lr}
0 & : x_1\leq a_{\rho}\,r\\
r^{\rho}\,C^{\frac{m-2}{2}}_{\rho}(x_1/r) & : x_1>a_{\rho}\,r
\end{array}
\right. 
\end{align} 
where $x:=(x_1,x_2,...,x_m)$ and $r\coloneqq |x|$. Note that $u_{\rho}$ is subharmonic in $\mathbb R^m$.
Dahlberg \cite{Dahlberg} then showed that for a subharmonic function $u$ in $\mathbb R^m, m\geq 3$ the following estimate is true
\begin{equation}
\label{eq:Dahlberg}
\vartheta(u)\leq \vartheta(u_{\rho}).
\end{equation} 
Identifying $\mathbb C^n\cong\mathbb R^{2n}$ then the inequality \eqref{eq:Dahlberg} holds also for functions $u$ that are subharmonic in $\mathbb C^n$. Note that now in the definition of $u_{\rho}$ we have $u_{\rho}(x)=r^{\rho}\,C^{n-1}_{\rho}(x_1/r)$ whenever $x_1<a_{\rho}r$ and $x:=(x_1,x_2,...,x_{2n}), r=|x|$. We calculate the characteristic functions $M(r,u_{\rho})$ and $T(r,u_{\rho})$:
\begin{align*}
M(r,u_{\rho})=\max\{u_{\rho}(x)\;:\;x\in\mathbb R^{2n}, |x|=r\}&=\max\{u_{\rho}(x)\;:\;x\in\mathbb R^{2n}, |x|=r, x_1>a_{\rho}r\}\\&=\max\{r^{\rho}\,C^{n-1}_{\rho}(x_1/r)\;:\;x\in\mathbb R^{2n}, |x|=r\}\\&=r^{\rho}\,C^{n-1}_{\rho}(1).
\end{align*}
Applying the formula for Nevalinna characteristic \eqref{Nev-main}, now on the unit sphere in $\mathbb R^{2n}$ yields
\begin{align*}
T(r,u_{\rho})=\int_{\mathbb S_{2n-1}}u^{+}_{\rho}(r\,\zeta)\,ds_{2n-1}(\zeta)=r^{\rho}\,\int_{\mathbb S_{2n-1}}C^{n-1}_{\rho}(\zeta_1)\cdot \chi_{\{\zeta_1>a_{\rho}\}}\,ds_{2n-1}(\zeta)
\end{align*}
where $\zeta:=(\zeta_1,\zeta_2,...,\zeta_{2n})\in\mathbb S_{2n-1}$ and $\chi_{\{\zeta_1>a_{\rho}\}}=1$ on the set $\{\zeta\in\mathbb S_{2n-1}\,:\,\zeta_1>a_{\rho}\}$ and vanishes elsewhere. Rewriting the last integral in terms of spherical coordinates implies
\begin{align*}
T(r,u_{\rho})&=\frac{r^{\rho}}{A(\mathbb S_{2n-1})}\,\int^{\pi}_0...\int^{\pi}_0\int_0^{\theta^*}C^{n-1}_{\rho}(\cos\theta_1)\,\sin^{2n-2}\theta_1\,\sin^{2n-3}\theta_2\,...\,\sin\theta_{2n-2}\,d\theta_1...\,d\theta_{2n-1}\\&
=\frac{r^{\rho}}{A(\mathbb S_{2n-1})}\,\int_0^{\theta^*}C^{n-1}_{\rho}(\cos\theta_1)\,\sin^{2n-2}\theta_1\,d\theta_1\,\int^{\pi}_0...\int^{\pi}_0\,\sin^{2n-3}\theta_2\,...\,\sin\theta_{2n-2}\,d\theta_2...\,d\theta_{2n-1}\\&
= \frac{A(\mathbb S_{2n-2})}{A(\mathbb S_{2n-1})}\,r^{\rho}\,\int_0^{\theta^*}C^{n-1}_{\rho}(\cos\theta_1)\,\sin^{2n-2}\theta_1\,d\theta_1
\end{align*}
where $A(\mathbb S_{k})$ is the area of the unit sphere in $\mathbb R^{k+1}$ and $\theta^*\coloneqq\arccos a_{\rho}$. As a result we obtain the ratio
$$\frac{M(r,u_{\rho})}{T(r,u_{\rho})}=C^{n-1}_{\rho}(1)\frac{A(\mathbb S_{2n-1})}{A(\mathbb S_{2n-2})}\cdot\Big(\int_0^{\theta^*}C^{n-1}_{\rho}(\cos\theta)\,\sin^{2n-2}\theta\,d\theta\Big)^{-1}.$$
Because the above quantity is independent of $r$ then we get
\begin{equation}
	\label{eq:u-phi}
	\vartheta(u_{\rho})=C^{n-1}_{\rho}(1)\frac{A(\mathbb S_{2n-1})}{A(\mathbb S_{2n-2})}\cdot\Big(\int_0^{\theta^*}C^{n-1}_{\rho}(\cos\theta)\,\sin^{2n-2}\theta\,d\theta\Big)^{-1}.
\end{equation}
To evaluate the last integral we make use of a recursion formula for Gegenbauer functions \cite{Kh99} given by 
\begin{equation}
\label{eq:recursion}
C^{n-1}_{\rho}(t)=\frac{1}{2^{n-1}\,(n-2)!}\,\frac{d^{n-2}}{dt^{n-2}}C^1_{\rho+n-2}(t)
\end{equation}
which yields 
\begin{align*}
\int_0^{\theta^*}C^{n-1}_{\rho}(\cos\theta)\,\sin^{2n-2}\theta\,d\theta&=\frac{1}{2^{n-1}\,(n-2)!}\,\int_0^{\theta^*}\frac{d^{n-2}}{d(\cos\theta)^{n-2}}C^1_{\rho+n-2}(\cos\theta)\,\sin^{2n-2}\theta\,d\theta\\&
=\frac{(-1)^{n-2}}{2^{n-1}\,(n-2)!}\,\int_0^{\theta^*}\frac{d^{n-2}}{d\theta^{n-2}}C^1_{\rho+n-2}(\cos\theta)\,\sin^{n}\theta\,d\theta.
\end{align*}
It is possible to find an explicit formula for $C^1_{\rho+n-2}(\cos\theta)$. Upon solving the differential equation \eqref{eq:Gegenbauer-diff} for $\gamma=1$ and $\rho\equiv \rho+n-2$ we get the expression 
\begin{equation}
\label{eq:solution}
C^1_{\rho+n-2}(\cos\theta)=\frac{1}{\cos\theta}\,\int^{\theta}_0\cos((\rho+n-1)t)\,dt=\frac{\sin((\rho+n-1)\theta)}{(\rho+n-1)\sin\theta},\quad 0<\theta<\pi.
\end{equation}
For illustration we consider the cases $n\in\{2,3\}$. 
\subsubsection{} When $n=2$ we have $C^1_{\rho}(\cos\theta)=\sin((\rho+1)\theta)/((\rho+1)\sin\theta)$ and the smallest $\theta>0$ for which $C^1_{\rho}(\cos\theta)=0$ is $\theta^*=\pi/(\rho+1)$. Then 
\begin{align*}
2\,(\rho+1)\,\int_0^{\theta^*}C^{1}_{\rho}(\cos\theta)\,\sin^{2}\theta\,d\theta&=\int_0^{\pi/(\rho+1)}\sin((\rho+1)\theta)\,\sin\theta\,d\theta\\&=\frac{1}{\rho}\,\sin\frac{\pi\rho}{\rho+1}-\frac{1}{\rho+2}\,\sin\frac{\pi(\rho+2)}{\rho+1}.
\end{align*}
By normalization formula \eqref{eq:normalization} we get $C^1_{\rho}(1)=\rho+1$, from the formula for the area of the unit sphere $A(\mathbb S_{k-1})=2\pi^{k/2}/\Gamma(k/2)$, we obtain $A(\mathbb S_{3})=2\pi^2, A(\mathbb S_{2})=4\pi$ therefore substituting in \eqref{eq:u-phi} yields
$$\vartheta(u_{\rho})=\pi(\rho+1)^2\cdot \Big(\frac{1}{\rho}\,\sin\frac{\pi\rho}{\rho+1}-\frac{1}{\rho+2}\,\sin\frac{\pi(\rho+2)}{\rho+1}\Big)^{-1}=\pi\rho\Big(1+\frac{\rho}{2}\Big)\,\frac{\rho+1}{\sin(\pi/(\rho+1))}.$$
Elementary calculations show that  
$$\frac{\rho+1}{\sin(\pi/(\rho+1))}\geq \frac{3}{\sin(\pi/3)}=2\sqrt{3}>e,\quad \rho\geq 2.$$
Therefore 
$$\vartheta(u_{\rho})>\pi\rho\Big(1+\frac{\rho}{2}\Big)\cdot e=e\,P_2(\rho),\quad\rho\geq 2.$$
In view of Remark \ref{r:3} and Theorem \ref{Maintheorem} then $\vartheta(u_{\rho})>K_2(\rho)$ for all $\rho\geq 2$. For the range $1<\rho<2$ one can compare directly with the explicit estimate for $K_2(\rho)$ given in \eqref{eq:K-2}.

\subsubsection{} When $n=3$ from the recursion formula \eqref{eq:recursion} it follows that
\begin{align*}
C^{2}_{\rho}(\cos\theta)=-\frac{1}{4\sin\theta}\cdot\frac{d}{ d\theta}C^1_{\rho+1}(\cos\theta)&=-\frac{1}{4}\Big(\frac{\cos((\rho+2)\theta)}{\sin^2\theta}-\frac{1}{\rho+2}\,\frac{\sin((\rho+2)\theta)\cos\theta}{\sin^3\theta}\Big)
\\&=\frac{1}{4}\Big((\rho+3)\,\frac{\sin((\rho+1)\theta)}{\sin^3\theta}-(\rho+1)\frac{\sin((\rho+3)\theta)}{\sin^3\theta}\Big).
\end{align*} 
Evaluating the integral 
\begin{align*}
\int_0^{\theta^*}C^{2}_{\rho}(\cos\theta)\,\sin^{4}\theta\,d\theta&=\frac{\rho+3}{4}\int^{\theta^*}_0\sin((\rho+1)\theta)\sin\theta\,d\theta-\frac{\rho+1}{4}\int^{\theta^*}_0\sin((\rho+3)\theta)\sin\theta\,d\theta\\&
=\frac{1}{8}\,\Big((\rho+3)\frac{\sin(\rho\,\theta^*)}{\rho}-2\,\sin((\rho+2)\theta^*)+(\rho+1)\frac{\sin((\rho+4)\theta^*)}{\rho+4}\Big).
\end{align*}
We obtain the estimate
$$\vartheta(u_{\rho})=3\pi\rho\Big(1+\frac{\rho}{2}\Big)\Big(1+\frac{\rho}{4}\Big)
\frac{(\rho+1)(\rho+3)}{m(\rho,\theta^*)}=3P_3(\rho)\,\frac{(\rho+1)(\rho+3)}{m(\rho,\theta^*)}
$$
where 
$$m(\rho,\theta^*):=(\rho+3)(\rho+4)\sin(\rho\,\theta^*)-2\rho(\rho+4)\,\sin((\rho+2)\theta^*)+\rho(\rho+1)\sin((\rho+4)\,\theta^*).$$
A rough approximation for $\theta^*$ is $\pi/(\rho+2)$ for large enough $\rho$. This implies upon substitution of this value in $m(\rho,\theta^*)$ that 
$$m(\rho,\theta^*)\approx ((\rho+3)(\rho+4)-\rho(\rho+1))\sin\Big(\frac{\rho}{\rho+2}\pi\Big)=6(\rho+2)\sin\Big(\frac{\rho}{\rho+2}\pi\Big).$$
Hence 
\begin{align*}
\vartheta(u_{\rho})=3P_3(\rho)\,\frac{(\rho+1)(\rho+3)}{m(\rho,\theta^*)}&\approx P_3(\rho)\frac{\rho+1}{2\cdot \sin(\rho\pi/(\rho+2))}>e^2\,P_3(\rho)
\end{align*}
for sufficiently large $\rho$.
Once more from Remark \ref{r:3} and Theorem \ref{Maintheorem} we have that for $\rho$ large enough $\vartheta(u_{\rho})>K_3(\rho)$. For small values of $\rho$ one has to explicitly compute $K_3(\rho)$ and make comparisons. But this is more laborious. For larger $n$ one can use the recursive relation \eqref{eq:recursion}, however computations get more complicated by the nature of the involved functions.

 \subsection*{Acknowledgements} I am grateful to Prof. B. Khabibullin for helpful comments and for sharing many materials related to the theory of subharmonic functions.
 \newline
 This work was in part supported by the DFG under Germany's Excellence Strategy – The Berlin Mathematics Research Center MATH+.
 
 \subsection*{Data availability} Data sharing not applicable to this article as no datasets were generated or analyzed during the current study.
\bibliographystyle{plain}
\bibliography{literature}

\end{document}